\documentclass[a4paper,10pt]{amsart}
\def\doublespace{}	
\usepackage{amsmath,amssymb,amsthm}
\usepackage[curve,matrix,arrow]{xy}
\usepackage{pb-diagram,pb-xy}
\usepackage[colorlinks=true]{hyperref}
\allowdisplaybreaks[4]
\catcode`\@ = 11
\newdimen\@InsertBoxMargin
\newcount\@numlines    
\newcount\@linesleft   
\def\ParShape{%
    \@numlines = 0
    \def\@parshapedata{ }
    \afterassignment\@beginParShape
    \@linesleft
}%
\def\@beginParShape{%
    \ifnum \@linesleft = 0
      \let\@whatnext = \@endParShape
    \else
      \let\@whatnext = \@readnextline
    \fi
    \@whatnext
}%
\def\@endParShape{%
    \global\parshape = \@numlines \@parshapedata
}%
\def\@readnextline#1 #2 #3 {
    \ifnum #1 > 0
      \bgroup  
        \dimen0 = \hsize
        \advance \dimen0 by -#2  
        \advance \dimen0 by -#3  
        \count0 = 0
        \loop
          \global\edef\@parshapedata{%
            \@parshapedata    
            #2                
            \space            
            \the\dimen0       
            \space            
          }%
          \advance \count0 by 1
          \ifnum \count0 < #1
        \repeat
      \egroup
      \advance \@numlines by #1
    \fi
    \advance \@linesleft by -1
    \@beginParShape
}%
\newbox\@boxcontent     
\newcount\@numnormal    
\newdimen\@framewidth   
\newdimen\@wherebottom  
\newif\if@byframe       
\@byframefalse
\def\InsertBoxC#1{%
  \leavevmode
  \vadjust{
    \vskip \@InsertBoxMargin
    \hbox to \hsize{\hss#1\hss}
    \vskip \@InsertBoxMargin
  }%
}%
\def\InsertBoxL#1#2{%
  \@numnormal = #1
  \setbox\@boxcontent = \hbox{#2}%
  \let\@side = 0
  \futurelet \@optionalparameter \@InsertBox
}
\def\InsertBoxR#1#2{%
  \@numnormal = #1
  \setbox\@boxcontent = \hbox{#2}%
  \let\@side = 1
  \futurelet \@optionalparameter \@InsertBox
}%
\def\@InsertBox{%
  \ifx \@optionalparameter [
    \let\@whatnext = \@@InsertBoxCorrection
  \else
    \let\@whatnext = \@@InsertBoxNoCorrection
  \fi
  \@whatnext
}%
\def\@@InsertBoxCorrection[#1]{%
  \ifx \@side 0
    \@@InsertBox{#1}{0}{{\the\@framewidth} 0cm}%
  \else
    \@@InsertBox{#1}{1}{0cm {\the\@framewidth}}%
  \fi
}%
\def\@@InsertBoxNoCorrection{%
  \@@InsertBoxCorrection[0]%
}%
\def\@@InsertBox#1#2#3{%
  \MoveBelowBox
  \@byframetrue
  \@wherebottom = \baselineskip
  \multiply \@wherebottom by \@numnormal
  \advance \@wherebottom by 2\@InsertBoxMargin
  \advance \@wherebottom by \ht\@boxcontent
  \advance \@wherebottom by \pagetotal
  \ifdim \pagetotal = 0cm
    \advance \@wherebottom by -\baselineskip  
  \fi
  \advance \@wherebottom by #1\baselineskip
  \@framewidth = \wd\@boxcontent
  \advance \@framewidth by \@InsertBoxMargin
  \bgroup  
    \ifdim \pagetotal = 0cm
      \dimen0 = \vsize
    \else
      \dimen0 = \pagegoal
    \fi
    \ifdim \@wherebottom > \dimen0
      \immediate\write16{+--------------------------------------------------------------+}%
      \immediate\write16{| The box will not fit in the page. Please, re-edit your text. |}%
      \immediate\write16{+--------------------------------------------------------------+}%
      \vrule width \overfullrule
    \fi
  \egroup
  \prevgraf = 0
  \vbox to 0cm{%
    \dimen0 = \baselineskip
    \multiply \dimen0 by \@numnormal
    \advance \dimen0 by -\baselineskip
    \setbox0 = \hbox{y}%
    \vskip \dp0
    \vskip \dimen0
    \vskip \@InsertBoxMargin
    \ifnum #2 = 1
      \vtop{\noindent \hbox to \hsize{\hss \box\@boxcontent}}%
    \else
      \vtop{\noindent \box\@boxcontent}%
    \fi
    \vss
  }%
  \vglue -\parskip
  \vskip -\baselineskip
  \everypar = {%
    \ifdim \pagetotal < \@wherebottom
      \bgroup  
        \dimen0 = \@wherebottom
        \advance \dimen0 by -\pagetotal
        \divide \dimen0 by \baselineskip
        \count1 = \dimen0
        \advance \count1 by 1
        \advance \count1 by -\@numnormal
        \ifnum #2 = 1
          \ParShape = 3
                      {\the\@numnormal}   0cm   0cm
                      {\the\count1}       0cm   {\the\@framewidth}
                      1                   0cm   0cm
        \else
          \ParShape = 3
                      {\the\@numnormal}   0cm                  0cm
                      {\the\count1}       {\the\@framewidth}   0cm
                      1                   0cm                  0cm
        \fi
      \egroup
    \else
      \@restore@    
    \fi
  }%
  \def\par{%
      \endgraf
      \global\advance \@numnormal by -\prevgraf
      \ifnum \@numnormal < 0
        \global\@numnormal = 0
      \fi
      \prevgraf = 0
  }%
}%
\def\MoveBelowBox{%
  \par
  \if@byframe
    \global\advance \@wherebottom by -\pagetotal
    \ifdim \@wherebottom > 0cm
      \vskip \@wherebottom
    \fi
    \@restore@
  \fi
}%
\def\@restore@{%
    \global\@wherebottom = 0cm
    \global\@byframefalse
    \global\everypar = {}%
    \global\let \par = \endgraf
    \global\parshape = 1 0cm \hsize
}%
\ifx \documentclass \@Dont@Know@What@It@Is@
\else
  \let \pageno = \c@page
\fi
\long\def\insertboxz[#1,#2,#3]{%
  \@numnormal = #1
  \setbox\@boxcontent = \hbox{#3}%
  \if #2r
    \let\@side = 1
    \@@InsertBoxCorrection[0]
  \else \if #2l
          \let\@side = 0
          \@@InsertBoxCorrection[0]
        \else \leavevmode
              \vadjust{
                \vskip \@InsertBoxMargin
                \hbox to \hsize{\hss#1\hss}
                \vskip \@InsertBoxMargin
              }%
        \fi
  \fi
}
\def\endinsertboxz{%
  \par
  \@restore@
}
\long\def\insertbox[#1,#2,#3,#4] {%
  \@numnormal = #1
  \setbox\@boxcontent = \hbox{#3}%
  \if #2r
    \let\@side = 1
    \@@InsertBoxCorrection[#4]
  \else \if #2l
          \let\@side = 0
          \@@InsertBoxCorrection[#4]
        \else \leavevmode
              \vadjust{
                \vskip \@InsertBoxMargin
                \hbox to \hsize{\hss#1\hss}
                \vskip \@InsertBoxMargin
              }%
        \fi
  \fi
}
\def\endinsertbox{%
  \par
  \@restore@
}

\catcode`\@ = 12
\theoremstyle{plain}
  \newtheorem{thm}{Theorem}[section]
  
  \newtheorem{prop}[thm]{Proposition}
  \newtheorem{cor}[thm]{Corollary}
  
\theoremstyle{definition}
  \newtheorem{defn}[thm]{Defninition}

\theoremstyle{remark}
  \newtheorem{rem}[thm]{Remark}

\def\TC#1{\def\args{#1}\ifx\args\empty\operatorname{\,\mathrm{TC}\,}\else\operatorname{\,\mathrm{TC}}(#1)\fi}
\def\TCM#1{\def\args{#1}\ifx\args\empty\operatorname{\,\mathrm{TC^{M}\,}}\else\operatorname{\,\mathrm{TC}^{M}}(#1)\fi}
\def\TCcuplen#1#2{\operatorname{\,\mathrm{Z}_{#2}}\ifx\empty#1\else(#1)\fi}
\def\TCwgt#1#2{\operatorname{wgt_{#2}}\ifx\empty#1\else(#1)\fi}
\def\tc#1{\operatorname{{tc}}\ifx\empty#1\else(#1)\fi}
\def\tcm#1{\operatorname{{tc}^{M}}\ifx\empty#1\else(#1)\fi}
\def\cat#1{\operatorname{{cat}}\ifx\empty#1\else(#1)\fi}
\def\catBB#1{\operatorname{{cat}^{B}_{B}}\ifx\empty#1\else(#1)\fi}
\def\catB#1{\operatorname{{cat}^{*}_{B}}\ifx\empty#1\else(#1)\fi}
\def\homeo{\approx}
\def\integral{\mathbb{Z}}
\def\field{\mathbb{F}}
\def\Int#1{\operatorname{Int}({#1})}
\def\ad{\operatorname{ad}}
\def\id{\operatorname{id}}
\def\midvert{\, \mathstrut \vrule \,}
\def\cell[#1||#2]{[#1|\{#2\}]}
\def\wgt#1{\operatorname{wgt}\ifx\empty#1\else(#1)\fi}
\def\Mwgt#1{\operatorname{Mwgt}\ifx\empty#1\else(#1)\fi}
\def\wgtBb#1{\operatorname{wgt_{B}}\ifx\empty#1\else(#1)\fi}
\def\wgtBB#1{\operatorname{wgt_{B}^{B}}\ifx\empty#1\else(#1)\fi}
\def\MwgtBb#1{\operatorname{Mwgt_{B}}\ifx\empty#1\else(#1)\fi}
\def\MwgtBB#1{\operatorname{Mwgt_{B}^{B}}\ifx\empty#1\else(#1)\fi}
\def\mwgtBb#1{\operatorname{wgt_{B}}\ifx\empty#1\else(#1)\fi}
\def\mwgtBB#1{\operatorname{wgt_{B}^{B}}\ifx\empty#1\else(#1)\fi}
\def\tcwgt#1{\operatorname{tc-wgt}\ifx\empty#1\else(#1)\fi}
\def\tcMwgt#1{\operatorname{tc-Mwgt}\ifx\empty#1\else(#1)\fi}
\def\cuplenBb#1{\operatorname{cup_{B}}\ifx\empty#1\else(#1)\fi}
\def\cuplenBB#1{\operatorname{cup_{B}^{B}}\ifx\empty#1\else(#1)\fi}
\def\proj{\operatorname{proj}}

\def\mathbold#1{\text{\boldmath{$#1$}}}
\begin{document}
\ifdefined\doublespace
\baselineskip20pt
\fi
%
%
\title{A short proof for $\tc{K}=4$} 
%
%
\author[Iwase]{Norio Iwase}
\email{iwase@math.kyushu-u.ac.jp\vskip-1ex}
\author[Sakai]{Michihiro Sakai}
\email{sakai@kurume-nct.ac.jp\vskip-1ex}
\author[Tsutaya]{Mitsunobu Tsutaya}
\email{tsutaya@math.kyushu-u.ac.jp}
%
%
\address[Iwase]
{Faculty of Mathematics,
 Kyushu University,
 Fukuoka 819-0395, Japan\vskip-1ex}
\address[Sakai]
{Liberal Arts, 
 National Institute of Technology, 
 Kurume College, 
 Fukuoka, 830-8555, Japan\vskip-1ex}
\address[Tsutaya]
{Faculty of Mathematics,
 Kyushu University,
 Fukuoka 819-0395, Japan}
%
%
\date{\today}
%
%
%
%
%
%
\begin{abstract}
We show a method to determine topological complexity from the fibrewise view point, which provides an alternative proof for $\tc{K}=4$, where $K$ denotes Klein bottle. 
\end{abstract}
%
%
\maketitle

\section{Introduction}
\vspace{-.5ex}
The topological complexity is introduced in \cite{MR1957228} by M. Farber for a space $X$ and is denoted by $\TC{X}$: 
$\TC{X}$ is the minimal number $m \!\geq\! 1$ such that $X{\times}X$ is covered by $m$ open subsets $U_{i} \ (1\!\leq\!i\!\leq\!m)$, each of which admits a continuous section $s_{i} : U_{i} \!\to\! \mathcal{P}(X)=\{u:[0,1] \!\to\! X\}$ for the fibration $\varpi : \mathcal{P}(X) \to X{\times}X$ given by $u \mapsto (u(0),u(1))$.
Similarly, the monoidal topological complexity of $X$ denoted by $\TCM{X}$ is the minimal number $m \!\geq\! 1$ such that $X{\times}X$ is covered by $m$ open subsets $U_{i} \!\supset\! \Delta{X} \ (1\!\leq\!i\!\leq\!m)$, each of which admits a section  $s_{i} : U_{i} \!\to\! \mathcal{P}(X)$ of $\varpi : \mathcal{P}(X) \to X{\times}X$ such that $s_{i}(x,x)$ is the constant path at $x$ for any $(x,x) \!\in\! U_{i}\cap\Delta{X}$. 
In this paper, we denote $\tc{X}=\TC{X}{-}1$ and $\tcm{X}=\TCM{X}{-}1$. 

Let $E\!=\!(E,B;p,s)$ be a fibrewise pointed space, i.e, $p : E \!\to\! B$ is a fibrewise space with a section $s : B \!\to\! E$.
For a fibrewise pointed space $E'\!=\!(E',B';p',s')$ and a fibrewise pointed map $f : E' \!\to\! E$, we have pointed and unpointed versions of fibrewise L-S category, denoted by $\catBB{f}$ and $\catB{f}$, respectively: 
$\catBB{f}$ is the minimal number $m\!\geq\!0$ such that $E'$ is covered by $(m{+}1)$ open subsets $U_{i}$ and $f_{i}\!=\!f\vert_{U_{i}}$ is fibrewise pointedly fibrewise compressible into $s(B)$, and $\catB{f}$ is the minimal number $m\!\geq\!0$ such that $E$ is covered by $(m{+}1)$ open subsets $U_{i}$ and $f_{i}\!=\!f\vert_{U_{i}}$ is fibrewise-unpointedly fibrewise compressible into $s(B)$.
We denote $\catBB{\id_{E}}\!=\!\catBB{E}$ and $\catB{\id_{E}}\!=\!\catB{E}$ (see \cite{MR2556074}).
Then by definition, $\tc{X} \!\leq\! \tcm{X}$ for a space $X$, $\catB{E} \!\leq\! \catBB{E}$ for a fibrewise pointed space $E$, and $\catB{f} \!\leq\! \catBB{f}$, $\catB{f} \!\leq\! \catB{E}$ and  $\catBB{f} \!\leq\! \catBB{E}$ for a fibrewise pointed map $f : E' \!\to\! E$.

In \cite{MR2556074}, the $m$-th fibrewise projective space $P_{\!\!B\,}^{m}\Omega_{B}E$ of the fibrewise loop space $\Omega_{B}(E)$ is introduced with a natural map $e^{E}_{m} : P_{\!\!B\,}^{m}\Omega_{B}E \!\to\! E$.
Using them, we characterise some numerical invariants:
firstly, the fibrewise cup-length $\cuplenBb{E;h}$ is $\max\left\{m{\geq}0 \midvert \exists_{\{u_{1},{\cdots},u_{m}\} \subset H^{*}(E,s(B))}  \ u_{1}{\cdots}u_{m}\not=0\right\}$. 
Secondly, the fibrewise categorical weight $\wgtBb{E;h}$ is the smallest number $m$ such that $e^{E}_{m} : P_{\!\!B\,}^{m}\Omega_{B}E \!\to\! E$ induces a monomorphism of generalised cohomology theory $h^{*}$.
Thirdly, the fibrewise module weight $\MwgtBb{E;h}$ is the least number $m$ such that $e^{E}_{m} : P_{\!\!B\,}^{m}\Omega_{B}E \!\to\! E$ induces a split monomorphism of generalised cohomology theory $h^{*}$ as an $h_{*}h$-module.
The latter two invariants are versions of categorical weight introduced by Rudyak \cite{MR1679849} and Strom \cite{MR1443893} whose origin is in Fadell-Husseini \cite{MR1317569}.
We obtain the following.
\vspace{-1ex}

\begin{thm}\label{thm:weight-catb}
$\cuplenBb{E;h}\le\wgtBb{E;h}\le\MwgtBb{E;h}\le\catB{E}\le\catBB{E}$.
\end{thm}\vspace{-1.5ex}
\vspace{-2ex}\begin{proof}
Let $\catB{E}\!=\! m$.
Then there is a covering of $E$ with $m{+}1$ open subsets $\{U_{i} \midvert 0 \!\leq\! i \!\leq\! m\}$ such that each $U_{i}$ can be compressed into $s(B) \!\subset\! E$.
So, there is an unpointed fibrewise homotopy of $\id : E \!\to\! E$ to a map $r_{i} : E \!\to\! E$ satisfying $r_{i}(U_{i}) \!\subset\! s(B)$, which gives an unpointed fibrewise compression of the fibrewise diagonal $\Delta_{B} : E \!\to\! \prod_{B}^{m+1} E$ into the fibrewise fat wedge $\prod_{B}^{[m+1]} E \!\subset\! \prod_{B}^{m+1} E$.
Since a continuous construction on a space can be extended on a cell-wise trivial fibrewise space by \cite{MR2427411}, the fibrewise projective $m$-space $P^{m}_{\!\!B\,}\Omega_{B}E$ has the fibrewise homotopy type of the fibrewise homotopy pull-back of $\Delta_{B} : E \to \prod^{m+1}_{B}E$ and the inclusion $\prod_{B}^{[m+1]} E \subset \prod_{B}^{m+1} E$.
Hence by James-Morris \cite{MR1130605}, we have a map $\sigma : E \to P^{m}_{\!\!B\,}\Omega_{B}E$ which is an unpointed fibrewise homotopy inverse of $e^{E}_{m} : P^{m}_{\!\!B\,}\Omega_{B}E \to E$, and hence we obtain $\MwgtBb{E;h}$ \!$\leq$\! $m$ \!$=$\! $\catB{E}$.
Combining this with \cite[Theorem 8.6]{MR2556074}\footnote{As is mentioned in \cite{MR2923451}, the equality of $\tcm{}$ and $\tc{}$ stated in \cite[Theorem 1.13]{MR2556074} is appeared to be an open statement. But the inequality in \cite[Theorem 8.6]{MR2556074} does not depend on the open statement.}, we obtain the theorem.
\end{proof}\vspace{-1ex}

From now on, we assume that $(E,B;p,s)$ is given by $E\!=\!X{\times}X$, $B\!=\!X$, $p\!=\!\proj_{1}: X{\times}X \!\to\! X$ and $s\!=\!\Delta : X \!\to\! X{\times}X$ the diagonal map, and so we have $\catB{E}=\tc{X}$ and $\catBB{E}=\tcm{X}$ by \cite{MR2556074,MR2923451}.
Hence we obtain the following by Theorem \ref{thm:weight-catb}.\vspace{-1ex}
\begin{thm}\label{thm:weight-tc}
$\wgtBb{E;h}\le\MwgtBb{E;h}\le\tc{X}\le\tcm{X}$\ for a space $X$.
\end{thm}\vspace{-1ex}

If $h$ is the ordinary cohomology with coefficients in a ring $R$, we write $\cuplenBb{E;h}$, $\wgtBb{E;h}$ and $\Mwgt{E;h}$ as $\cuplenBb{E;R}$, $\wgtBb{E;R}$ and $\Mwgt{E;R}$, respectively. 
We might disregard $R$ later in this paper, if $R=\field_{2}$ the prime field of characteristic $2$.

As an application, we give an alternative proof of a result recently announced by several authors.
Let $K_{q}$ be the non-orientable closed surface of genus $q\!\geq\!1$, and denote $K\!=\!K_{2}$.\vspace{-1ex}
\begin{thm}[Cohen-Vandembroucq \cite{CV}]
\label{thm_tc(K)}
For $q \!\geq\! 2$, we have $\wgt{K_{q}}=\Mwgt{K_{q}}=\tc{K_{q}}=\tcm{K_{q}}=4$ and $\TC{K_{q}}=\TCM{K_{q}}%
=5$. 
\end{thm}\vspace{-1ex}

\vspace{-1ex}
\begin{cor}\label{choi}
The fibration $S^{1} \!\hookrightarrow\! K \!\to\! S^{1}$ is a counter example to the conjecture saying $\TC{E} \leq \TC{F}{\,\times\!}\TC{B}$ for a fibration $F \!\to\! E \!\to\! B$, since $\TC{K}\!=\!5$ and $\TC{S^{1}}=2$.
\end{cor}\vspace{-1ex}

\section{Fibrewise Resolution of Klein Bottle}
\vspace{-.5ex}
For $q \!\geq\! 1$, $\pi_{1}(K_{q})$ is given by 
$\pi_{1}^{q}\!=\!\langle{b,b_{1},\ldots,b_{q-1} \midvert b_{1}^{2}{\cdots}b_{q-1}^{2}{=}b^{2}}\rangle$.
We know that $K_{q}$ is a CW complex with one $0$-cell $\ast$, $q$ $1$-cells $b, b_{1}, \ldots, b_{q-1}$ and one $2$-cell $\sigma_{q}$. 

\ifx\undefined\doublespace
\InsertBoxR{0}{%
\begin{picture}(90,80)(-4,0)
\multiput(10,10)(0,65){2}{\line(1,0){65}}
\multiput(10,10)(65,0){2}{\line(0,1){65}}
\multiput(48,22)(5,-5){2}{\line(1,1){12}}
\put(38,10){\vector(1,0){10}}
\put(33,10){\vector(1,0){10}}
\put(75,35){\vector(0,1){10}}
\put(10,50){\vector(0,-1){10}}
\put(33,75){\vector(1,0){10}}
\put(38,75){\vector(1,0){10}}
\put(40,-2){$b$}
\put(0,40){$a$}
\put(80,40){$a$}
\put(40,80){$b$}
\put(78,78){$*$}
\put(40,40){\makebox(6,6)[cc]{$\sigma$}}
\multiput(75,10)(-65,0){2}{\circle*{3}}
\multiput(75,75)(-65,0){2}{\circle*{3}}
\end{picture}
}
\fi
For $a\!=\!b_{1}b^{-1}\!$, we know $\pi_{1}^{2}\!=\!\{a^{k}b^{\ell} \midvert k,\ell\!\in\!\integral\}$ with a relation $aba\!=\!b$.
Let us denote $\varepsilon(\ell)\!=\!\frac{1{-}(-1)^{\ell}}{2}$, which is either $0 \,\text{or}\, 1$, to obtain \,$a^{k_{1}}b^{\ell_{1}}a^{k_{2}}b^{\ell_{2}}\!=\!a^{k_{1}+k_{2}+2\varepsilon(\ell_{1})k_{2}}b^{\ell_{1}+\ell_{2}}$, \,$b^{\pm1}(a^{k}b^{\ell})b^{\mp1}\!=\!a^{-k}b^{\ell}$ and \,$a^{\pm1}(a^{k}b^{\ell})a^{\mp1}\!=\!a^{k\pm2\varepsilon(\ell)}b^{\ell}$.
We denote $\bar{\tau}\!=\!\tau^{-1}$ to simplify expressions. 

We know the multiplication of $\pi_{1}^{2}\!=\!\pi_{0}(\Omega{K_{2}})$ is inherited from the loop addition.
Hence the natural equivalence $\Omega{K_{2}} \!\to\! \pi_{1}^{2}$ is an $A_{\infty}$-map, since a discrete group has no non-trivial $A_{\infty}$-structure on a given multiplication.

Let $E_{q}\!=\!(E_{q},B_{q};p_{q},s_{q})$ be the fibrewise pointed space, where $E_{q}\!=\!K_{q}{\times}K_{q}, \,B_{q}\!=\!K_{q}, \,p_{q}\!=\!\proj_{1}:K_{q}{\times}K_{q}\!\to\!K_{q}$ and $s_{q}\!=\!\Delta:K_{q}\!\to\!K_{q}{\times}K_{q}$. 
When $q\!=\!2$, we abbreviate $E_{2}$, $K_{2}$, $\sigma_{2}$ and $\pi_{1}^{2}$ as $E$, $K$, $\sigma$ and $\pi$, respectively in this paper.

\ifdefined\doublespace
\InsertBoxR{0}{%
\begin{picture}(90,70)(-4,5)
\multiput(10,10)(0,65){2}{\line(1,0){65}}
\multiput(10,10)(65,0){2}{\line(0,1){65}}
\multiput(48,22)(5,-5){2}{\line(1,1){12}}
\put(38,10){\vector(1,0){10}}
\put(33,10){\vector(1,0){10}}
\put(75,35){\vector(0,1){10}}
\put(10,50){\vector(0,-1){10}}
\put(33,75){\vector(1,0){10}}
\put(38,75){\vector(1,0){10}}
\put(40,-2){$b$}
\put(0,40){$a$}
\put(80,40){$a$}
\put(40,80){$b$}
\put(78,78){$*$}
\put(40,40){\makebox(6,6)[cc]{$\sigma$}}
\multiput(75,10)(-65,0){2}{\circle*{3}}
\multiput(75,75)(-65,0){2}{\circle*{3}}
\end{picture}
}
\fi
Let $\widetilde{K}\!=\!\underset{\mathbold{a} \in K}{\textstyle\bigcup}\, \pi_{1}(K;\mathbold{a},\ast) \!\to\! K$ be the universal covering space, and $\widehat{K} \!=\! \widetilde{K} {\times_{\ad}}\pi \!\to\! K$ be the associated covering space, where `$\ad$' is the equivalence relation on $\widetilde{K} {\times}\pi$ given by $([\kappa{\cdot}\lambda],g) \sim ([\kappa],hgh^{-1})$ for $g, \,h\!=\![\lambda] \in \pi$ and $[\kappa] \in \pi_{1}(K;\mathbold{a},\ast)$.
We regard $\widehat{K}\!=\!\underset{\mathbold{a} \in K}{\textstyle\bigcup}\, \pi_{1}(K,\mathbold{a})$.
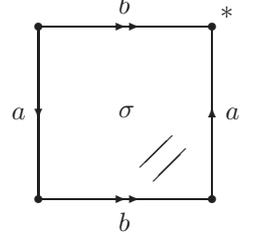
\vspace{-1ex}
\begin{prop}
$P_{\!\!B}^{m}\Omega_{B}E \simeq_{B} P_{\!\!B}^{m}\widehat{K}$ for all $m \!\geq\! 1$.
\end{prop}\vspace{-1ex}
\vspace{-1.5ex}\begin{proof}
For $[\gamma]\!=\!g\!\in\!\pi$, we denote by $\Omega^{g}_{B}E$ and $\widehat{K}^{g}$ the connected components of $\gamma\!\in\! \Omega_{B}E$ and $([\ast],g) \!\in\! \widetilde{K}{\times}_{\ad}\pi \!=\! \widehat{K}$, respectively.
Then the image of $\pi_{1}(\Omega^{g}_{B}E)$ in $\pi_{1}(K)$ is the centralizer of $g$, which is the same as $\pi_{1}(\widehat{K}^{g})$. Thus, there is a lift $\widehat{\Omega^{g}_{B}p} : \Omega^{g}_{B}E \to \widehat{K}^{g}$ of $\Omega^{g}_{B}p\!=\!\Omega_{B}p\vert_{\Omega^{g}_{B}E} : \Omega^{g}_{B}E \to K$ whose restriction to the fibre on $\mathbold{a}$ is the natural map $: \Omega(K,\mathbold{a})\cap\Omega^{g}_{B}E \!\to\! \pi_{1}(K,\mathbold{a})\cap\widehat{K}^{g}$.
Hence we obtain a lift $\widehat{\Omega_{B}p} : \Omega_{B}E \!\to\! \widehat{K}$ of $\Omega_{B}p : \Omega_{B}E \!\to\! K$ given by $\widehat{\Omega_{B}p}\vert_{\Omega^{g}_{B}E} \!=\! \widehat{\Omega^{g}_{B}p}$, whose restriction to the fibre on $\mathbold{a}$ is the natural map $: \Omega(K,\mathbold{a}) \!\to\! \pi_{1}(K,\mathbold{a})$.
Moreover, the restriction of $\widehat{\Omega_{B}p}$ to each fibre is a pointed homotopy equivalence since $K$ is a $K(\pi,1)$ space.
Then by Dold \cite{MR0073986}, $\widehat{\Omega_{B}p} : \Omega_{B}E \!\to\! \widehat{K}$ is a fibrewise homotopy equivalence.
Here, since the section $: K \!\to\! \Omega^{e}_{B}E$ of $\Omega^{e}_{B}p : \Omega^{e}_{B}E\!\to\!K$ given by trivial loops is a fibrewise cofibration, $\widehat{\Omega_{B}p}$ is a fibrewise pointed homotopy equivalence by James \cite{MR1361889}.
Moreover, $\widehat{\Omega_{B}p}$ is a fibrewise $A_{\infty}$-map since each fibre of $\widehat{K} \!\to\! K$ is a dicrete set.
Thus $P_{\!\!B}^{m}\Omega_{B}E \simeq_{B} P_{\!\!B}^{m}\widehat{K}$, $m \!\geq\! 1$.
\end{proof}\vspace{-1ex}

Firstly, the cell structure of $K$ is given as follows: let $\Lambda_{0}\!=\!\{*\}, \,\Lambda_{1}\!=\!\{{a,b}\}, \,\Lambda_{2}\!=\!\{\sigma\}$.
\par\vskip.5ex\noindent
\hfil$
K = \underset{0 \leq k \leq 2}{\textstyle\bigcup}\, \underset{\eta \in \Lambda_{k}}{\textstyle\bigcup}\, e^{k}_{\eta} = e^{0}_{*} \cup e^{1}_{a} \cup e^{1}_{b} \cup e^{2}_{\sigma}.
$\hfil
\par\vskip1ex\noindent
From now on, $e^{k}_{\eta}$ will be denoted by $[\eta]$ for $\eta \!\in\! \Lambda_{k}$, which is in the chain group $\integral\Lambda\!=\!\integral\{*,a,b,\sigma\}$, $\Lambda=\Lambda_{0} \cup \Lambda_{1} \cup \Lambda_{2}$.
The boundary of $[\eta]$ for $\eta \in \Lambda_{k}$ is expressed in $\integral\Lambda$ as follows:
\par\vskip0ex\noindent
\hfil$
\partial[\eta]=[\partial\eta],\quad 
\partial{*} = 0,\ \  
\partial{a} = 0,\ \  
\partial{b} = 0
 \ \ \text{and} \ \ 
\partial{\sigma} = 2a, 
$\hfil
\par\vskip0ex\noindent

Secondly, $P^{m}\pi$ is a $\Delta$-complex in the sense of Hatcher \cite{MR1867354}:
\par\vskip0ex\noindent
\hfil$
P^{m}\pi = \underset{0 \leq n \leq m}{\textstyle\bigcup}\, \underset{\omega=(g_{1},\ldots,g_{n}) \in \pi^{n}}{\textstyle\bigcup} e^{n}_{\omega},
$\hfil
\par\vskip1.0ex\noindent
In this paper, $e^{n}_{\omega}$ will be denoted by $[\omega]$ or $[g_{1}|\cdots|g_{n}]$ for $\omega \!=\! (g_{1},\ldots,g_{n})$ which is in the chain group $\underset{n=0}{\overset{m}{\oplus}}\otimes^{n}\integral{\pi}$ \!$\cong$\! $\underset{n=0}{\overset{m}{\oplus}}\integral{\pi}^{n}$.
The boundary of $[\omega]$ is expressed as follows:
\par\vskip.5ex\noindent
\hfil$
\partial{[\omega]} = [\partial\omega],\quad \partial\omega = \underset{i=0}{\overset{n}{\textstyle\sum}}(-1)^{i}\partial_{i}\omega,\quad
\partial_{i}{\omega} = \text{$
\left\{\begin{array}{l}
\partial_{0}{\omega} = (g_{2},\ldots,g_{n}),\quad i \!=\! 0,
\\[1ex]
(g_{1},\ldots,g_{i-1},g_{i}g_{i+1},g_{i+2},\cdots,g_{n}),\quad 0 \!<\! i \!<\! n,
\\[1ex]
\partial_{n}{\omega} = (g_{1},\ldots,g_{n-1}),\quad i \!=\! n,
\end{array}\right.$}
$\hfil
\par\vskip1.5ex\noindent
which coincides with the chain complex of $m$-th filtration of the bar resolution of $\pi$.

\ifx\undefined\doublespace
\InsertBoxR{0}{%
\begin{picture}(170,85)(-48,10)
\multiput(10,10)(0,65){2}{\line(1,0){65}}
\multiput(10,10)(65,0){2}{\line(0,1){65}}
\put(38,10){\vector(1,0){10}}
\put(33,10){\vector(1,0){10}}
\put(75,35){\vector(0,1){10}}
\put(10,50){\vector(0,-1){10}}
\put(33,75){\vector(1,0){10}}
\put(38,75){\vector(1,0){10}}
\put(0,40){\makebox(6,6)[cr]{\small$\cell[a||ba\omega\bar{a}\bar{b}]$}}
\put(80,40){\makebox(6,6)[cl]{\small$\cell[a||\omega]$}}
\put(40,-2){\makebox(6,6)[cc]{\small$\cell[b||a\omega\bar{a}]$}}
\put(40,82){\makebox(6,6)[cc]{\small$\cell[b||\omega]$}}
\put(40,40){\makebox(6,6)[cc]{$\cell[\sigma||\omega]$}}
\put(77,0){\makebox(6,6)[cl]{\small$\cell[*||a\omega\bar{a}]$}}
\put(3,0){\makebox(6,6)[cr]{\small$\cell[*||ba\omega\bar{a}\bar{b}]$}}
\put(77,79){\makebox(6,6)[cl]{\small$\cell[*||\omega]$}}
\put(3,79){\makebox(6,6)[cr]{\small$\cell[*||b\omega\bar{b}]$}}
\multiput(75,10)(-65,0){2}{\circle*{3}}
\multiput(75,75)(-65,0){2}{\circle*{3}}
\end{picture}
}
\fi
For $\tau \!\in\! \Lambda_{1}$, and $\omega \!\in\! {\pi}^{n}$, $\cell[\bar{\tau}||\omega]$ represents the same product cell as $\cell[\tau||\bar{\tau}\omega\tau]$ with orientation reversed, and we have $\cell[\bar{\tau}||\omega]\!=\!-\cell[\tau||\bar{\tau}\omega\tau]$, where $\bar{\tau}(g_{1},\ldots,g_{n})\tau\!=\!(\bar{\tau}g_{1}{\tau},\ldots,\bar{\tau}g_{n}{\tau})$.
To observe this, let us look at the end point of $\tau$, where the fibre lies: 
A $1$-cell $\tau$ is a path $\tau : I\!=\![0,1] \to K$ which has a lift to a path $\tilde\tau : I \to \widehat{K}$ with an initial data $[\lambda] \in \pi_{1}(K,\tau(1))$ given by $\tilde\tau(t)=[\tau_{t}{\cdot}\lambda{\cdot}\tau_{t}^{-1}] \in \pi_{1}(K,\tau(t))$, where we denote $\tau_{t}(s)=\tau(t{+}(1{-}t)s)$.

\ifdefined\doublespace
\InsertBoxR{0}{%
\begin{picture}(170,85)(-48,10)
\multiput(10,10)(0,65){2}{\line(1,0){65}}
\multiput(10,10)(65,0){2}{\line(0,1){65}}
\put(38,10){\vector(1,0){10}}
\put(33,10){\vector(1,0){10}}
\put(75,35){\vector(0,1){10}}
\put(10,50){\vector(0,-1){10}}
\put(33,75){\vector(1,0){10}}
\put(38,75){\vector(1,0){10}}
\put(0,40){\makebox(6,6)[cr]{\small$\cell[a||ba\omega\bar{a}\bar{b}]$}}
\put(80,40){\makebox(6,6)[cl]{\small$\cell[a||\omega]$}}
\put(40,-2){\makebox(6,6)[cc]{\small$\cell[b||a\omega\bar{a}]$}}
\put(40,82){\makebox(6,6)[cc]{\small$\cell[b||\omega]$}}
\put(40,40){\makebox(6,6)[cc]{$\cell[\sigma||\omega]$}}
\put(77,0){\makebox(6,6)[cl]{\small$\cell[*||a\omega\bar{a}]$}}
\put(3,0){\makebox(6,6)[cr]{\small$\cell[*||ba\omega\bar{a}\bar{b}]$}}
\put(77,79){\makebox(6,6)[cl]{\small$\cell[*||\omega]$}}
\put(3,79){\makebox(6,6)[cr]{\small$\cell[*||b\omega\bar{b}]$}}
\multiput(75,10)(-65,0){2}{\circle*{3}}
\multiput(75,75)(-65,0){2}{\circle*{3}}
\end{picture}
}
\fi
Thirdly, since $\Omega_{B}E$ is fibrewise $A_{\infty}$-equivalent to $\widehat{K}$, 
$P_{\!\!B}^{m}\Omega_{B}E$ is fibrewise pointed homotopy equivalent to $P_{\!\!B}^{m}\hat{K}$.
A $k{+}n$-cell of $P_{\!\!B}^{m}\Omega_{B}E \simeq_{B} P_{\!\!B}^{m}\widehat{K} = \widetilde{K} {\times_{\ad}} P^{m}\pi$ is described as a product cell of a $k$-cell $[\eta]$ in $K$ and a $\Delta$ $n$-cell $[\omega]$ in $P^{m}\pi$, and is denoted by $e^{n+k}_{(\eta;\omega)} \homeo \Int{\Box^{k}}{\times}\Int{\Delta^{n}}$. 
\par\vskip1ex\noindent
\hfil$
P_{\!\!B\,}^{m}\Omega_{B}E \simeq_{B} 
P_{\!\!B\,}^{m}\widehat{K} = \underset{0 \leq n \leq m}{\textstyle\bigcup}\, 
\underset{\omega\in {\pi}^{n}}{\textstyle\bigcup}\, \left(e^{n}_{(*;\omega)} \cup e^{n+1}_{(b;\omega)} \cup e^{n+1}_{(b_{1};\omega)} \cup e^{n+2}_{(\sigma;\omega)}\right).
$\hfil
\par\vskip1.5ex\noindent

In this paper, $e^{n+k}_{(\eta;\omega)}$ will be denoted by $\cell[\eta||\omega]$ or $\cell[\eta||g_{1}|\cdots|g_{n}]$, for $(\eta;\omega) \!=\! (\eta;g_{1},\ldots,g_{n}) \!\in\! \Lambda_{k}{\times}{\pi}^{n}$, in the chain group $C^{*}(P_{\!\!B}^{m}\widehat{K};\integral)=\!\underset{n=0}{\overset{m}{\oplus}}\,\integral\,\Lambda_{0}{\times}{\pi}^{n} \oplus \underset{n=1}{\overset{m+1}{\oplus}}\,\integral\,\Lambda_{1}{\times}{\pi}^{n-1} \oplus \underset{n=2}{\overset{m+2}{\oplus}}\,\integral\,\Lambda_{2}{\times}{\pi}^{n-2}$.

Let $[\omega]\!=\![g_{1}|g_{2}|\cdots|g_{n}]$ be a $\Delta$ $n$-cell in $P^{m}\widehat{K}$ with $g_{i} \!\in\! \pi_{1}(K,\tau(1))$.
Then the boundary of a product cell $\cell[\tau||\omega]$ of $\omega$ with a $1$-cell $[\tau]$ of $K$ is the union of cells \,$\cell[\tau||\partial_{i}\omega]$, \,$0 \!\leq\! i \!\leq\! n$, \,$[\omega]$ and $[\tau\omega\bar{\tau}]\!=\![{\tau}g_{1}\bar{\tau}|{\tau}g_{2}\bar{\tau}|\cdots|{\tau}g_{n}\bar{\tau}]$.
Similarly, the boundary of a product cell $\cell[\sigma||\omega]$ of $\omega$ with a $2$-cell $[\sigma]$ of $K$ is the union of cells \,$\cell[\sigma||\partial_{i}\omega]$, \,$0 \!\leq\! i \!\leq\! n$, \ifdefined\,$\cell[a||\omega]$, \,$\cell[b||\omega]$, $\cell[a||ba\omega\bar{a}\bar{b}]$ and $\cell[b||a\omega\bar{a}]$\else\,$\cell[b||\omega]$, \,$\cell[b||b\omega\bar{b}]$, $\cell[b_{1}||\omega]$ and $\cell[b_{1}||b_{1}\omega\bar{b}_{1}]$\fi.

Hence the boundary formula of a $\Delta$ cell of $P_{\!\!B}^{m}\widehat{K}$ in the chain group modulo $2$ is given as follows, where we denote $m\!\underset{(p)}=\!n$\vspace{-2ex} if $m$ is equal to $n$ modulo $p$ for $m, n \!\in\! \integral$ and $p \!\geq\!2$.
\begin{prop}\label{prop:boundary-formula}
\begin{enumerate}
\item $\partial\cell[\tau||\omega] \underset{(2)}= \cell[*||\omega] - \cell[*||\tau\omega\bar{\tau}] - \cell[\tau||\partial\omega] \ \ \text{for $\tau \!\in\! \Lambda_{1}$\, and \,$\omega \!\in\! {\pi}^{n}$}$,\vspace{-.5ex} where we denote $\cell[\tau||\partial\omega]=\underset{i=0}{\overset{n}{\textstyle\sum}}(-1)^{i}\cell[\tau||\partial_{i}\omega]$.
\vspace{.5ex}\item 
$\partial\cell[\sigma||\omega] \underset{(2)}= \cell[a||\omega] + \cell[a||ba\omega\bar{a}\bar{b}] - \cell[b||\omega] + \cell[b||a\omega\bar{a}] + \cell[\sigma||\partial\omega]$ \ for \,$\omega \!\in\! {\pi}^{n}$,\vspace{-.5ex} where we denote $\cell[\sigma||\partial\omega]=\underset{i=0}{\overset{n}{\textstyle\sum}}(-1)^{i}\cell[\sigma||\partial_{i}\omega]$.
\vspace{-1ex}
\end{enumerate}
\end{prop}\vspace{-2ex}

\section{Topological Complexity of non-orientable surface}
\vspace{-.5ex}
Since $P^{\infty}\pi \simeq K$, we have $H^{*}(P^{\infty}\pi) =\field_{2}\{1,x,y,z\}$ with $z\!=\!xy\!=\!yx\!=\!x^{2}$, where $x, \,y$ are dual to $[a], \,[b]$, respectively, the generators of $H_{1}(P^{\infty}\pi) \cong \field_{2}[a]\oplus\field_{2}[b]$.
We regard $x$ and $y$ are in $Z^{1}(P^{\infty}\pi)$ and $z\!=\!x \!\cup\! y$ is in $Z^{2}(P^{\infty}\pi)$.
A simple computation shows that $[a^{k}b^{\ell}]$ is homologous to $k[a]\!+\!\ell[b]$ in $Z_{1}(P^{\infty}\pi)$, and we have $x[a^{k}b^{\ell}]=k$ and $y[a^{k}b^{\ell}]=\ell$.
By definition of a cup product in a chain complex, we obtain the following equality:
\par\vskip.5ex\noindent\hfil$
z[a^{k_{1}}b^{\ell_{1}}|a^{k_{2}}b^{\ell_{2}}]
=(x \!\cup\! y)[a^{k_{1}}b^{\ell_{1}}|a^{k_{2}}b^{\ell_{2}}]
=x[a^{k_{1}}b^{\ell_{1}}]{\cdot}y[a^{k_{2}}b^{\ell_{2}}]
=k_{1}\ell_{2}\quad \text{in $P^{m}\pi$}, 
$\hfil\par\vskip.5ex\noindent
where we denote $x|_{P^{m}\pi}$, $y|_{P^{m}\pi}$ and $z|_{P^{m}\pi}$ again by $x$, $y$ and $z$, respectively.\vspace{-1ex}
\begin{prop}\label{prop:iso}
\begin{enumerate}
\item\label{prop:iso-1}
$e^{K}_{m} : P^{m}\pi \!\hookrightarrow\! P^{\infty}\pi \!\overset{\simeq}\to\! K$ induces, up to dimension $2$ in the ordinary $\field_{2}$-cohomology, a monomorphism if $m \!\geq\! 2$,\vspace{.5ex} and an isomorphism if $m \!\geq\! 3$.
\item\label{prop:iso-2}
$e^{E}_{m} : P^{m}_{\!\!B}\widehat{K} \!\hookrightarrow\! P_{\!\!B\,}^{\infty}\widehat{K} \overset{\simeq}\to E$ induces, up to dimension $4$ in the ordinary $\field_{2}$-cohomology, a monomorphism if $m \!\geq\! 4$, and an isomorphism if $m \!\geq\! 5$.
\end{enumerate}
\end{prop}\vspace{-1ex}
\vspace{-1.5ex}\begin{proof}
Since $P^{m}\pi$ is the $m$-skeleton of $P^{\infty}\pi$, the pair $(P^{\infty}\pi,P^{m}\pi)$ is $m$-connected, and so is the fibrewise pair $(P_{\!\!B\,}^{\infty}\widehat{K},P_{\!\!B\,}^{m}\widehat{K})$ over $K$.
It implies the proposition.
\end{proof}\vspace{-1ex}

By Proposition \ref{prop:iso} (\ref{prop:iso-1}), we can easily see the following propostion.\vspace{-1ex}
\begin{prop}\label{prop:top-dimension}
The cocycle $z$ represents the generator of $H^{2}(P^{m}\pi) \cong \field_{2}$\, for \,$m \geq 3$.
\end{prop}\vspace{-1ex}

Associated with the filtration $\{F_{i}(m)\!=\!p_{m}^{-1}(K^{(i)})\}$ of $P_{\!\!B\,}^{m}\widehat{K} \simeq_{B} P_{\!\!B\,}^{m}\Omega_{B}E$, given by the CW filtration $\{*\}\!=\!K^{(0)}\!\subset\!K^{(1)}\!\subset\!K^{(2)}\!=\!K$ of $K$ with $K^{(1)}\!= \{*\} \cup e^{1}_{(a)} \cup e^{1}_{(b)} \homeo S^{1} \!\vee\! S^{1}$, we have Serre spectral sequence $E_{r}^{*,*}(m)=E_{r}^{*,*}(P_{\!\!B\,}^{m}\widehat{K})$ converging to $H^{*}(P_{\!\!B\,}^{m}\widehat{K})$ with $E_{1}^{p,q}(m) \cong H^{p+q}(F_{p}(m),F_{p-1}(m)) \cong H^{p}(K^{(p)},K^{(p-1)};H^{q}(P^{m}\pi))$ the cohomology with local coefficients.

From now on, we denote $\alpha\!=\!(a^{k_{1}}b^{\ell_{1}})$, \,$\tau\!=\!(a^{k_{1}}b^{\ell_{1}},a^{k_{2}}b^{\ell_{2}})$ and $\omega\!=\!(a^{k_{1}}b^{\ell_{1}},a^{k_{2}}b^{\ell_{2}},a^{k_{3}}b^{\ell_{3}})$.
Let functions $: [a^{k_{1}}b^{\ell_{1}}|\cdots|a^{k_{n}}b^{\ell_{n}}] \mapsto k_{i}$ and $\ell_{i}$ by $(k_{i})$ and $(\ell_{i})$, respectively for $1 \!\leq\! i \!\leq\! n$.
Then for a function $f : \integral^{2n} \!\to\! \integral$, we obtain a function $(f(\{k_{i}\},\{\ell_{i}\})) : [a^{k_{1}}b^{\ell_{1}}|\cdots|a^{k_{n}}b^{\ell_{n}}] \mapsto f(\{k_{i}\},\{\ell_{i}\})$.
By Proposition \ref{prop:iso}, $H^{4}(P_{\!\!B\,}^{5}\widehat{K}) \!\cong\! \field_{2}$ is generated by $(e^{E}_{5})^{*}([z{\otimes}z])$, which comes from $E_{1}^{2,2}(5) \cong H^{4}(F_{2}(5),F_{1}(5))$ for dimensional reasons.
By the isomorphism $H^{4}(F_{2}(5),F_{1}(5)) \!\cong\! H^{2}(P^{5}\pi)$,  $(e^{E}_{5})^{*}[z{\otimes}z]$ corresponds to  $[z] \!\in\! H^{2}(P^{5}\pi)$ by Proposition \ref{prop:top-dimension}, and hence a representing cocycle $w \!\in\! Z^{4}(P_{\!\!B\,}^{5}\widehat{K})$ of $(e^{E}_{5})^{*}[z{\otimes}z]$ can be chosen as a homomorphism defined by the formulae
\par\vskip.5ex\noindent
\hfil$
w{\cell[\sigma||\tau]} = z[\tau] = k_{1}\ell_{2},
\quad w\vert_{F_{1}(5)}=0.
$\hfil
\par\vskip.5ex\noindent
When $3\!\leq\!m\!\leq\!5$, we denote $w\vert_{P_{\!\!B\,}^{m}\widehat{K}}$ again by $w \!\in\! Z^{4}(F_{2}(m),F_{1}(m))$, which is representing a generator of $E_{1}^{2,2}(m)$.
Furthermore, $[w]\!\not=\!0$ in $E_{\infty}^{2,2}(m)$ if $m \!\geq\! 4$ by Proposition \ref{prop:iso}.

Our main goal is to show $[w]\!=\!0$ in $H^{*}(P_{\!\!B\,}^{3}\widehat{K})$:
we remark here that $\varepsilon(\ell) \underset{(2)}= \ell$ for $\ell\!\in\!\integral$, since \,$\varepsilon(\ell) \!=\! 0$ \!${\iff}$\! $(-1)^{\ell}\!=\!1$ \!${\iff}$\! $\ell$\, is even.

Firstly, let us introduce a numerical function given by the floor function.\vspace{-1ex}

\begin{defn}
$t(m)=\lfloor{\frac{m}{2}}\rfloor
$ for \,$m \in \integral$.
\end{defn}\vspace{-1ex}
Then we have $t(0)\!=\!0$ and we obtain the following.
\vspace{-1ex}
\begin{prop}\label{app-prop:numerical-functions}
\begin{enumerate}
\item\label{numericalfunction2}
$t(-m) \underset{(2)} = t(m)\!+\!m$, 
\item\label{numericalfunction3b}
$t(m{+}n{+}2\ell) \underset{(2)}= t(m){+}t(n){+}mn{+}\ell$,\vspace{-1ex} \ for \,$m, n, \ell \!\in\! \integral$.
\end{enumerate}
\end{prop}\begin{proof}
This proposition can be obtained by strait-forward calculations, and so we left it to the reader.
\end{proof}\vspace{-1ex}

\vspace{-1ex}
\begin{cor}\label{app-prop:numerical-functions-coboundary}
\begin{enumerate}
\item $t(k_{1})[\partial\tau] 
\!=\! t(k_{2}) 
\!+\! t(k_{1}{+}k_{2}{+}2\varepsilon(\ell_{1})k_{2})
\!+\! t(k_{1}) 
\!=\! (\ell_{1}{+}k_{1})k_{2}$,
\par\item
$(k_{1}t(k_{2}))[\partial\omega] 
\!=\! k_{2}t(k_{3}) 
\!+\! (k_{1}{+}k_{2})t(k_{3}) 
\!+\! k_{1}t(k_{2}{+}k_{3}{+}2\varepsilon(\ell_{2})k_{3})
\!+\! k_{1}t(k_{2}) 
\!=\! k_{1}(\ell_{2}{+}k_{2})k_{3}$.
\end{enumerate}
\end{cor}\vspace{-1ex}

Secondly, we introduce an element $u \!\in\! C^{3}(P_{\!\!B\,}^{3}\widehat{K})$ given by the formulae below:
\par\vskip1.5ex\noindent
\hfil$\begin{array}{l}
u\cell[*||\omega]=k_{1}t(k_{2})\ell_{3}k_{3} + k_{1}(\ell_{2}k_{3}{+}k_{2}\ell_{3}{+}k_{2})t(k_{3}),
\\[2ex]
u\cell[a||\tau] = 0, \ \ 
u\cell[b||\tau] = (k_{1}t(k_{2}))[\tau] 
\ \ \text{and} \ \ 
u\cell[\sigma||\alpha] = 0.
\end{array}$\hfil
\par\vskip1.5ex\noindent
Then $\delta{u}$ enjoys the following formulae by Propositions \ref{prop:boundary-formula}, \ref{app-prop:numerical-functions} and Corollary \ref{app-prop:numerical-functions-coboundary} in $C^{*}(P_{\!\!B\,}^{3}\widehat{K})$: 
\begin{enumerate}
\item
$\begin{array}[t]{l}
(\delta{u})\cell[\sigma||\tau]
\underset{(2)}= 
u\cell[a||\tau]
+ u\cell[a||a^{-k_{1}-2\varepsilon(\ell_{1})}b^{\ell_{1}}|a^{-k_{2}-2\varepsilon(\ell_{2})}b^{\ell_{2}}]
\\[0ex]\hskip25mm\quad%
+\, u\cell[b||\tau] 
\,+\, u\cell[b||a^{k_{1}+2\varepsilon(\ell_{1})}b^{\ell_{1}}|a^{k_{2}+2\varepsilon(\ell_{2})}b^{\ell_{2}}] 
+ u\cell[\sigma||\partial{\tau}]
\\[1ex]\hskip0mm\quad%
\underset{(2)}= 
0 + k_{1}(t(k_{2}) {+} t(k_{2}{+}2\varepsilon(\ell_{2})))
+ 0 \underset{(2)}= 
k_{1}\varepsilon(\ell_{2}) \underset{(2)}= k_{1}\ell_{2} = w\cell[\sigma||\tau].
\end{array}$
\par\vskip1ex
\item
$\begin{array}[t]{l}\textstyle 
(\delta{u})\cell[a||\omega]
\underset{(2)}=
u\cell[*||\omega]$ $+$ $u\cell[*||a^{k_{1}+2\varepsilon(\ell_{1})}b^{\ell_{1}}|a^{k_{2}+2\varepsilon(\ell_{2})}b^{\ell_{2}}|a^{k_{3}+2\varepsilon(\ell_{3})}b^{\ell_{3}}] + u\cell[a||\partial\omega]
\\\quad\textstyle
\underset{(2)}=
k_{1}(t(k_{2}){+}t(k_{2}{+}2\varepsilon(\ell_{2})))\ell_{3}k_{3}
+ 
k_{1}(\ell_{2}k_{3}{+}k_{2}\ell_{3}{+}k_{2})(t(k_{3}) {+} t(k_{3}{+}2\varepsilon(\ell_{3})))
+ 
0
\\[1ex]\quad\textstyle 
\underset{(2)}= 
k_{1}\varepsilon(\ell_{2})\ell_{3}k_{3}
+ 
k_{1}(\ell_{2}k_{3}{+}k_{2}\ell_{3}{+}k_{2})\varepsilon(\ell_{3})
\underset{(2)}= 0 = w\cell[a||\omega].
\end{array}$
\par\vskip1ex
\item
$\begin{array}[t]{l}\textstyle 
(\delta{u})\cell[b||\omega] 
\underset{(2)}=
u\cell[*||\omega] + u\cell[*||a^{-k_{1}}b^{\ell_{1}}|a^{-k_{2}}b^{\ell_{2}}|a^{-k_{3}}b^{\ell_{3}}] + u\cell[b||\partial\omega]
\\\quad\textstyle
\underset{(2)}=
k_{1}(t(k_{2}) {+} t(-k_{2}))\ell_{3}k_{3}
+ 
k_{1}(\ell_{2}k_{3}{+}k_{2}\ell_{3}{+}k_{2})(t(k_{3}){+}t(-k_{3})) + (k_{1}t(k_{2}))[\partial\omega]
\\[1ex]\quad\textstyle 
\underset{(2)}=
k_{1}k_{2}\ell_{3}k_{3}
+ 
k_{1}(\ell_{2}k_{3}{+}k_{2}\ell_{3}{+}k_{2})k_{3} + k_{1}(\ell_{2}{+}k_{2})k_{3}
\underset{(2)}= 0 = w\cell[b||\omega].
\end{array}$
\end{enumerate}
Thus we obtain that $\delta{u} \!\underset{(2)}=\! w$ in $C^{*}(P_{\!\!B\,}^{3}\widehat{K})$, which enables us to show the following.\vspace{-1.5ex}
\begin{thm}\label{thm:main}
$\tcm{K} = \tc{K} = \wgtBb{E} = \wgtBb{z{\otimes}z} = 4$.
\end{thm}\vspace{-2ex}
\vspace{-1ex}\begin{proof}
By the above arguments, we have $(e^{E}_{3})^{*}(z{\otimes}z)=[w]=[\delta{u}]=0$ in $H^{*}(P_{\!\!B\,}^{3}\widehat{K})$, and hence 
$0\!\not=\!z{\otimes}z \in \ker\,(e^{E}_{3})^{*}$ which implies $\wgtBb{E}\geq\wgtBb{z{\otimes}z}\geq4$.
On the other hand, Theorem \ref{thm:weight-tc} implies $\mwgtBb{E}\leq\tc{K}\leq\tcm{K}\leq2\cat{K}\leq2\dim{K}\!=\!4$.
It implies the theorem.
\end{proof}\vspace{-1ex}

\begin{rem}
Let $u_{0} \!\in\! C^{2}(P_{\!\!B\,}^{2}\widehat{K})$ and $w_{0} \in C^{3}(P_{\!\!B\,}^{2}\widehat{K})$ be as follows:
\par\vskip.5ex\noindent
\hfil$
\begin{array}[t]{l}
u_{0}\cell[*||\tau]=(t(k_{1})\ell_{2}k_{2} + (\ell_{1}k_{2}{+}k_{1}\ell_{2}{+}k_{1})t(k_{2}))[\tau], \quad
u_{0}\cell[a||\alpha]=0, \ \ 
\\[1ex]
u_{0}\cell[b||\alpha]=t(k_{1})[\alpha], \quad
u_{0}\cell[\sigma||*]=0; \qquad 
w_{0}{\cell[\sigma||\alpha]} = y[\alpha] = \ell_{1}, \ \ w_{0}\vert_{F_{1}(2)}=0.
\end{array}
$\hfil
\par\vskip1.0ex\noindent
Then we can observe $\delta(u_{0})\!\underset{(2)}=\!w_{0}$ and $[w_{0}]\!=\!0$ in $H^{*}(P_{\!\!B\,}^{2}\widehat{K})$, which would imply $\wgtBb{z{\otimes}y}=3$.
\vspace{-1ex}
\end{rem}

Let $q \!\geq\!2$.
Then by sending $b$ to $b$, $b_{1}$ to $ab$, and all other $b_{i}$'s to $1$, $1\!<\!i\!<\!q$, we obtain a homomorphism $\phi_{q} : \pi^{q}_{1} \!\to\! \pi$, since $(ab)^{2}\!=\!b^{2}$ in $\pi$.
Then $\phi_{q}$ induces maps $B\phi_{q} : K_{q}\!=\!B\pi_{q} \!\to\! B\pi\!=\!K$ and $P^{m}_{\!\!B\,}\widehat{\phi_{q}} : P^{m}_{\!\!B\,}\widehat{K_{q}} \!\to\! P^{m}_{\!\!B\,}\widehat{K}$ such that $e^{E_{q}}_{m}{\circ}P^{m}_{\!\!B\,}\widehat{\phi_{q}} = (\phi_{q}{\times}\phi_{q}){\circ}e^{E}_{3}$.
Since $\phi_{q}^{*} : H^{2}(K) \!\to\! H^{2}(K_{q})$ is an isomorphism, $z_{q}\!:=\!\phi_{q}^{*}(z)$ is the generator of $H^{2}(K_{q}) \!\cong\! \field_{2}$.
Hence $(e^{E_{q}}_{3})^{*}(z_{q}{\otimes}z_{q}) \!=\! (e^{E_{q}}_{3})^{*}{\circ}(\phi_{q}{\times}\phi_{q})^{*}(z{\otimes}z) \!=\! (P^{3}_{\!\!B\,}\widehat{\phi_{q}})^{*}{\circ}(e^{E}_{3})^{*}(z{\otimes}z)$ $=$ $0$ by Theorem \ref{thm:main}, and we obtain $4 \leq \wgtBb{z_{q}{\otimes}z_{q}} \leq \wgtBb{E_{q}}$.
It implies the following.\vspace{-1ex}
\begin{thm}
$\tcm{K_{q}} = \tc{K_{q}} = \wgtBb{E_{q}} = \wgtBb{z_{q}{\otimes}z_{q}} = 4$\, for all $q \geq 2$.
\vskip-1ex
\end{thm}

%
%
\bibliographystyle{alpha}
\bibliography{2017tc}

\medskip

\end{document}
\end